\documentclass{amsart}
\usepackage{color, lineno}
\usepackage[hidelinks]{hyperref}
\usepackage{amsmath, amssymb, amsfonts, amstext, amsthm, textcomp, comment,bbm}
\usepackage{nccmath,enumitem}
\usepackage[left=3.5cm, right=3.5cm]{geometry}
\usepackage{chngcntr}
\usepackage{float}
\usepackage{tikz, pgfplots}
\pgfplotsset{compat=1.18}
\usetikzlibrary{calc,angles,positioning,intersections,quotes,decorations.markings}
\usepackage{tkz-euclide}

\newtheorem{theorem}{Theorem}[section]
\newtheorem{lemma}[theorem]{Lemma}

\newtheorem{proposition}[theorem]{Proposition}

\theoremstyle{definition}
\newtheorem{definition}[theorem]{Definition}
\newtheorem{example}[theorem]{Example}

\theoremstyle{remark}
\newtheorem{remark}[theorem]{Remark}
\numberwithin{equation}{section}

\newcommand{\N}{\mathbb{N}}

\newcommand{\R}{\mathbb{R}}

\newcommand{\calH}{\mathcal{H}}

\newcommand{\suchthat}{\;\ifnum\currentgrouptype=16 \middle\fi|\;}

\begin{document}

\title{Invertible positive maps that are not automorphisms}

\author{Pavankumar Raickwade}
\address{Department of Mathematics, Indian Institute of Technology Madras, Chennai, Tamil Nadu, India}
\curraddr{}
\email{prraickwade@gmail.com}
\thanks{}

\author{K. C. Sivakumar}
\address{Department of Mathematics, Indian Institute of Technology Madras, Chennai, Tamil Nadu, India}
\curraddr{}
\email{kcskumar@iitm.ac.in}
\thanks{}



\date{}

\dedicatory{}
            
 

\begin{abstract}
    Let $X$ be a real normed vector space with a cone $K\subseteq X$ satisfying either (i) $K$ is closed with non-empty interior or (ii) $K$ has non-zero extremals or (iii) $K$ is closed and $X$ is a Banach space. In this short note, we provide a method to construct an invertible linear map $T\colon X\to X$ such that $T[K]\subseteq K$ but $T^{-1}[K]\not\subseteq~K$. In particular, we show that, for every cone automorphism $S\colon X\to X$, there exists a rank one perturbation of $S$ which is positive and invertible, but does not have a positive inverse. We provide examples from four diverse situations. 
\end{abstract}

    
\maketitle

\section{Introduction}

Throughout this article, $X$ denotes a real vector space and $X'$ denotes the algebraic  dual of $X$, i.e. the space of all linear functionals on $X$. The space of all invertible (bijective) linear operators on $X$ is denoted by $L_{inv}(X)$.  A subset $K\subseteq X$ is called a {\it cone} if $K$ is closed under addition, $\lambda K\subseteq K$ for every $\lambda \geq 0$ and $K\cap (-K)=\{0\}$. A linear map $T\colon X\to X$ is said to be {\it positive} (with respect to the cone $K$) if $T[K]\subseteq K$. The set of all positive operators is denoted by $\pi(K)$. Cones and positive operators have been extensively investigated, see for example, \cite{DBLP:books/siam/BermanP94, positivity_book_Aliprantis}. An invertible linear map $T\colon X\to X$ such that both $T$ and $T^{-1}$ are positive is called a {\it cone automorphism}. The set of all cone automorphisms is denoted by $Aut(K)$. 

The problem of characterizing cone automorphisms has been considered by many authors for several examples of closed cones. For instance, the Lorentz cone in finite dimensional spaces is considered in \cite{Sznajder2024-ju}, the cone of copositive matrices is considered in \cite{shitov2019linearmappingspreservingcopositive} and independently in \cite{GOWDA20133862}, the general theory of group of automorphisms in a finite dimensional space is covered in \cite{HORNE1978111}, etc. When studying rank one perturbations of automorphisms in certain select examples, it was observed that the resulting operators are positive, invertible but their inverses were not positive. Motivated by this observation we set out to study the general problem. 

In this article we provide a method to construct a linear map $T\in \pi(K)\cap L_{inv}(X)\setminus Aut(K)$ for a cone $K\subseteq X$. More specifically, we show, in Theorem~\ref{thm::main} that, for a closed cone $K$ with non-empty interior (when $X$ has a norm) and for any $S\in Aut(K)$, there exists a rank one perturbation of $S$ which is positive, invertible but not an automorphism. Analogous results are obtained in Theorem~\ref{thm::ext_non_empty} (for cones possessing nonzero extremals) and in Theorem~\ref{thm::Banach-space-case} (for closed cones in Banach spaces). We provide many illustrative examples.

\section{Preliminaries}

Let $X$ be a vector space. A partial order $\leq$ on $X$ is called a {\it vector space order} if the following compatibility relations hold: 
\begin{align*}
	\forall x,y,z \in X&\colon \ x \leq y \Longrightarrow x+z\leq y+z,\\
	\forall x,y\in X, \lambda\geq 0&\colon \ x \leq y \Longrightarrow \lambda x\leq \lambda y.
\end{align*}
In this case, we call $X$ an {\it ordered vector space}\footnote{In some literature, `ordered vector space' corresponds to a total order. However, we consider only partially ordered vector spaces, and prefer the terminology `ordered vector space' over `partially ordered vector space'.}. Recall that, a set $K\subseteq X$ is called a {\it cone} if $K$ is closed under addition, $\lambda K\subseteq K$ for every $\lambda \geq 0$, and $K\cap (-K)=\{0\}$. 
If $\leq$ is a vector space order on $X$, then  
\begin{equation} 
	\label{eq:cone}
	X_+:=\{x \in X\mid x\geq 0\}
\end{equation} is a cone. On the other hand, if $K\subseteq X$ is a cone, then a vector space order on $X$ is given by \begin{equation}\label{cone_order_relation}
    x\leq y ~:\Longleftrightarrow~  y-x\in K,
\end{equation}
and for $X_+$ as in \eqref{eq:cone}  one obtains $X_+=K$. Hence, there is a one-to-one correspondence between cones and vector space orders, and we denote an ordered vector space by $(X,X_+)$ in case we want to specify the corresponding cone. If $X$ has a norm $\|\cdot\|$ then $(X, X_+, \|\cdot \|)$ is called an {\it ordered normed space}. If $X$ is a Banach space then $(X, X_+, \|\cdot\|)$ is called an {\it ordered Banach space}.  The articles \cite{BARKER1981263,cones_and_duality, Anke_book} are excellent sources for information on cones and partially ordered vector spaces.

\begin{definition}
    Let $(X,X_+)$ be an ordered vector space. A linear map $T\colon X\to X$ is called
\begin{itemize}
    \item [-]
        {\it positive} if $x\in X_+ \Longrightarrow Tx\in X_+$. 
    \item [-]
        {\it automorphism} if $T$ is bijective and both $T, T^{-1}$ are positive.
\end{itemize}
\end{definition}

For a cone $K\subseteq X$, define the set $K':=\{f\in X'\mid \forall x\in K;~f(x) \geq 0\}$. The following result is well known, and will be useful in the sequel.
\begin{proposition}[{\cite[Proposition 1.5.5]{Anke_book}}]\label{prop::xinKifff(x)>0}
    Let $(X,K,\|\cdot\|)$ be an ordered normed space with $K$ being closed. Then, $x\in K$ if and only if $f(x)\geq 0$ for every $f\in K'$.
\end{proposition}

Let $X$ be an ordered vector space. An element $x\in K$ is called an {\it extremal} if, for every $0\leq z\leq x$, there exists $0\leq\lambda\leq 1$ such that $z=\lambda x$. Note that every positive vector below an extremal is also an extremal. The set of all extremals of a cone $K$ is denoted by $\text{ext}(K)$.  We use the following result repeatedly.

\begin{lemma}\label{lem::inv_image_of_extremal_is_extremal}
    Let $T\colon X\to X$ be an injective, positive linear map. If $T(x)$ is an extremal of $T[K],$ then $x$ is an extremal of $K$. In particular, if $T\in Aut(K)$, then $Tx \in \emph{ext}(K)$ if and only if $x \in \emph{ext}(K).$
\end{lemma}
\begin{proof}
    Since $T$ is positive, $T[K]$ is also a cone. Let $T(x)$ be an extremal of $T[K]$ and $0\leq z\leq x$. Then $0\leq T(z)\leq T(x)$. Hence, $T(z)=\lambda T(x)$ for some $\lambda\geq 0$, and so $z=\lambda x$. Thus, $x$ is an extremal of $K$, proving the first part. The second part follows by applying the first, to both $T$ and $T^{-1}$.
\end{proof}

\begin{definition}
    Let $(X,X_+)$ be an ordered vector space. An element $u\in X_+\setminus\{0\}$ is called an {\it order unit} if, for  every $x\in X$, there is $\lambda\in (0,\infty)$ such that $x\leq \lambda u$.
\end{definition}

The following results are well known. We provide proofs for the ready reference.
\begin{proposition}\label{prop::int--o.u.}
    Let $(X,K, \|\cdot\|)$ be a  ordered normed vector space with $\text{int}(K)\neq \emptyset$.
    \begin{enumerate}[label=\upshape(\roman*)]
        \item \label{prop::int--o.u.-(i)}
            An element $u\in K\setminus \{0\}$ is an order unit if and only if $u\in \text{int}(K)$.
        \item \label{prop::int--o.u.-(v)}
            Let $f\in K'\setminus \{0\}$ and $u\in \text{int}(K)$. Then $f(u)>0$.
        \item \label{prop::int--o.u.-(ii)}
            Let $u\in \text{int}(K)$ and $v\in K$ such that $v\geq u$. Then $v\in \text{int}(K)$.
        \item \label{prop::int--o.u.-(iv)}
             Let $u\in \text{int}(K)$ and $\alpha>0$. Then $\alpha u\in \emph{int}(K)$,
        \item \label{prop::int--o.u.-(iii)}
            Let $S\in Aut(K)$. Then $S(u)\in \text{int}(K)$ if and only if $u\in \text{int}(K)$
    \end{enumerate}
\end{proposition}
\begin{proof}
\begin{enumerate}[label=\upshape(\roman*)]
    \item 
         Refer \cite[Proposition 1.5.11 (i)]{Anke_book} for the proof of \ref{prop::int--o.u.-(i)}. 
    \item 
        Let $x\in X$. By (i), we know that there exits $\lambda >0$ be such that $-\lambda u\leq x\leq \lambda u$. Since $f\in K'$, we get $-\lambda f(u)\leq f(x)\leq \lambda f(u)$. Hence, if $f(u)=0,$ then $f(x)=0$. Since $x$ is arbitrary, we get that $f=0$. A contradiction.
    \item 
         By (i), this is straightforward.
    \item 
        Fix $r>0$ such that $B_r(u):=\{y\in X\mid \|y-u\|<r\}\subseteq K.$ As $K$ is a cone, it follows that $B_{\alpha r}(\alpha u)\subseteq K$ for every $\alpha>0$.
    \item 
         It is enough to show that $S(u)\in \text{int}(K)$ whenever $u\in \text{int}(K)$, or equivalently, $S(u)$ is an order unit, whenever $u$ is an order unit. Let $u\in K\in \text{int}(K)$ and $x\in X$. By (i), we know that $u$ is an order unit, hence there exits $\lambda >0$ such that $S^{-1}(x)\leq \lambda u$. Since $S$ is positive, we get $x\leq \lambda S(u)$. Therefore $S(u)\in \text{int}(K)$.\qedhere
\end{enumerate}
\end{proof}


\section{Cones with non-empty interior}

We assume throughout that $X$ is a real vector space with a cone $K$. For a linear functional $f$ on $X$, we denote its null space by $\ker (f)$. 

Fix $u\in K \setminus \{0\},~ f\in K'\setminus \{0\}$ and $S\in Aut(K)$. Define a linear map $T\colon X\to X$~as
\begin{equation}\label{genT}
     T(x):=Sx+f(x) u, \quad \forall x\in X.
\end{equation}
In particular, for $S=I$, the identity operator, 
\begin{equation}\label{parT}
    T(x)=x+f(x) u, \quad \forall x\in X.
\end{equation}
Observe that, by definition, $T$ is a positive map. First, in Theorem \ref{thm::invertibility}, we show that the map $T$ as defined in \eqref{genT} is invertible, and then, in Theorem~\ref{thm::main}, we show that, if $u\in \text{int}(K)$, then $T^{-1}$ is not positive. In Theorem \ref{thm::ext_non_empty}, under the assumption that $\text{ext}(K)\neq \{0\}$, we obtain a similar result via a specific choice of $f$ and $u$.

\begin{theorem}\label{thm::invertibility}
    Let $K\subseteq X$ be a cone in a real vector space $X$. For any $u\in K\setminus \{0\}$ and $f\in K'\setminus \{0\}$, the map $T$ as defined in \eqref{genT} is positive and invertible.
\end{theorem}
\begin{proof}
    As noted earlier, $T$ is positive. Suppose $T(x)=S(x)+f(x)u=0, x\in X$. Applying $S^{-1},$ and then $f$, we get $f(x)+f(x)f(S^{-1}u)=0$, i.e. $(1+f(S^{-1}u))f(x)=0$. Since $f\in K'$ and $S^{-1}\geq 0$, we get $f(x)=0$. This means that $S(x)=0$, and hence $x=0,$ proving the injectivity of $T$. Next, for $y\in X$, choose $x:=S^{-1}y-\lambda S^{-1}u$, where $\lambda:=\frac{f(S^{-1}y)}{1+f(S^{-1}u)}$. Then 
    \[
        T(x)=y-\lambda u+(f(S^{-1}y)-\lambda f(S^{-1}u))u=y-\lambda u+\lambda u=y. \qedhere
    \]
\end{proof}

\begin{remark}\label{spform}
In the proof above, we have shown in particular that, if $T$ is as defined in \eqref{parT}, then $T^{-1}(y)=y-\lambda u,$ where $\lambda:=\frac{f(y)}{1+f(u)}$. We shall make use of this formula in Example \ref{lp}.
\end{remark}

Before stating our main results, we state and prove a few auxiliary results. For $A\subseteq X$, $\text{int}(A)$ denotes the topological interior of $A$.
\begin{lemma}\label{lemma::boundary_lemma}
    Let $A\subseteq X$ be a closed and convex subset with non-empty interior in a normed vector space $X$. Let $u\in \emph{int}(A)$ and $v\notin A$. Consider the map $g\colon [0,1]\to X$, defined as $g(t)=(1-t)u+tv$ for $t\in [0,1]$. Then there exists $c\in (0,1)$ such that $g(c)\in \partial(A)$, the boundary of $A$. Moreover, 
    \begin{center}
        $g(t)\in A$ if $t\leq c,$ while $g(t) \notin A$ for $t>c.$
    \end{center}
\end{lemma}
\begin{minipage}{0.5\textwidth}
    A geometric description of $g(c):$
\end{minipage}
\hspace*{-0.5in}
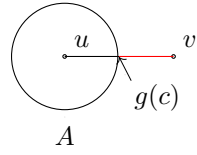
\begin{minipage}{0.5\textwidth}
     \begin{center}
        \begin{tikzpicture}[scale=0.8]
            \draw (0,0) circle (25pt);
            \draw (0,-1) circle (0.0001mm) node[anchor=north]{$A$};
            \draw (1, -0.3) circle (0.0001mm) node[anchor=north west]{$g(c)$};
            \draw [very thin, ->] (1.1,-0.4)--(0.9, -0.01);
            \draw (0,0) circle (0.3mm) node[anchor=south west]{$u$};
            \draw (1.8,0) circle (0.3mm) node[anchor=south west]{$v$};
            \draw (0,0)--(0.9,0);
            \draw[color=red] (0.9,0)--(1.8,0);
        \end{tikzpicture}
    \end{center}
\end{minipage}
   
\begin{proof}
    First, observe that $g$ is continuous. Let $c:=\inf g^{-1}(A^c)$. Clearly, $g([0,c])\subseteq A$. Since $A^c$ is an open set, $g^{-1}(A^c)$ is open. Note that $1 \in g^{-1}(A^c).$ Thus $(\delta,1] \subseteq g^{-1}(A^c)$ for some $\delta>0$. Hence, $c<1$. Since $g(0)\in \text{int}(A)$, there exists $\epsilon>0$ such that $B_\epsilon(g(0))\subseteq A$. As $g$ is continuous, there exists $r>0$ such that $g[0, r)\subseteq B_\epsilon (g(0))$. Thus $0<c<1$. Also, notice that, for any $t\in (0,1)$ and $\lambda\in (0, t)$, $g(\lambda)=(1-\frac{\lambda}{t})u+\frac{\lambda}{t}g(t)$. Since $A$ is convex and $u\in A,$ if $g(t)\in A$, then $g((0,t))\subseteq A$. Therefore $g((c,1])\subseteq A^c$. Choose sequences $s_n$ and $t_n$ in $[0,1]$ converging to $c,$ such that $s_n<c<t_n,$ for every~$n$. Then $g(s_n)\in A \not \ni g(t_n),$ and so $g(c)\in \partial(A)$.
\end{proof}

For an ordered vector space $(X,K)$, if $T\in \pi(K)\cap L_{inv}(X)\setminus Aut(K)$, then there exists $x_0\in X\setminus K$ such that $T(x_0)\in K$. Notice that $x_0\in X\setminus (K\cup(-K))$. Thus, a necessary condition for the existence of $T\in \pi(K)\cap L_{inv}(X)\setminus Aut(K)$ is that $X\setminus (K\cup(-K))\neq \emptyset.$ This condition is guaranteed whenever $K$ is a closed cone in a normed vector space $X$ with $\dim X\geq 2$, as we show next.

\begin{lemma}\label{lem::uncomparable_elements}
    Let $K\subseteq X$ be a closed cone in a normed vector space $X$. If  $K\cup (-K)=X$, then $X$ is one dimensional. In other words, if $\dim X\geq 2$, then there exists $x\in X$ such that $x\nleq 0$ and $x\ngeq 0$.
\end{lemma}
\begin{proof}
    First, we show that, for any $x,y\in X$
    \begin{equation}\label{eq::Archimedean}
        (\forall n\in \N\colon nx\leq y) \implies x\leq 0.
    \end{equation} 
    Indeed, since $K$ is closed, $\frac{1}{n}y-x\in K$ for every $n\in \N$ implies  $-x\in K$. Suppose $X=K\cup (-K),$ so that any two elements of $X$ are comparable. Let $x, y\in X$ be such that $x\neq 0$ and $y\neq 0$. Without loss of generality, let $x\geq 0$. Then, by \eqref{eq::Archimedean}, there exists $N\in \N$ such that $Nx\nleq y$, and so $Nx\geq y$. Hence,  $P:=\{\alpha>0\mid \alpha x\geq y\}\neq \emptyset.$ Consider $c:=\inf P$. There are two cases to consider; $cx\geq y$ and $cx\leq y$. 

    Let $cx\leq y$. Since $c=\inf P$, for every $n\in \N$, there exists $\alpha_n\in P$ such that $c+\frac{1}{n}>\alpha_n>c$, and so $(c+\frac{1}{n})x\geq \alpha_n x\geq y$. Since $K$ is closed, we get $cx\geq y$. Now, let $cx\geq y$. For every $n\in \N$, we get $(c-\frac{1}{n})x\ngeq y$. As any two elements are comparable, it follows that $(c-\frac{1}{n})x\leq y$ for every $n\in \N$. Since $K$ is closed, this implies $cx\leq y$. Thus, in any case, we get $cx=y$. Hence,  $x$ and $y$ are linearly dependent, showing that $X$ is one dimensional.
\end{proof}

\begin{theorem}\label{thm::main}
    Let $X$ be a real normed vector space and let $K\subseteq X$ be a closed cone with $\emph{int} (K)\neq \emptyset$, and $S\in Aut(K)$. Then, for any $f\in K'\setminus\{0\}$ and $u\in \emph{int}(K)$,  the map $T\colon X\to X$, given by,
    \begin{equation*}
        T(x):=Sx+f(x) u, \quad \forall x\in X,
    \end{equation*}
    is positive and invertible, but $T^{-1}$ is not positive.
\end{theorem}
\begin{proof}
    Suppose there exists $x\in \partial K$ such that $f(x)> 0$, then $S(x) + f(x) u\geq f(x) u$. Thus, by Proposition~\ref{prop::int--o.u.} \ref{prop::int--o.u.-(ii)} and \ref{prop::int--o.u.-(iv)}, we get $T(x)\in \text{int}(K)$. By Proposition~\ref{prop::int--o.u.}\ref{prop::int--o.u.-(iii)}, it then follows that $T\notin Aut(K)$. Hence, it is enough to find $x\in \partial K$ such that $f(x)> 0$. Suppose, on the contrary, that $f(x)=0$ for every $x\in \partial K$. Take any $v\in X\setminus K$. By Lemma~\ref{lemma::boundary_lemma}, there exists $c\in (0,1)$ such that $(1-c)u + cv\in \partial K$. Hence $f(v)=\frac{c-1}{c}f(u)$. By Proposition~\ref{prop::int--o.u.}\ref{prop::int--o.u.-(v)}, we get $f(v)<0$. Hence $\ker f\subseteq K$. This is a contradiction to $K$ being a cone. Thus there exists $x_0\in \partial K$ such that $f(x_0)>0$.
\end{proof}

\begin{remark} The following observations are pertinent. 
\begin{enumerate}
    \item 
        The assumption $u\in \text{int}(K)$, in Theorem \ref{thm::main} is indispensable. For example, let $X=\R^n, K:=\R^n_+, S:=I$ and $u=f:=e_1$. Then $T(x)=(2x_1, x_2,\ldots, x_n), ~x \in \R^n,$ is an automorphism.
    \item 
        The assumption $\text{int}(K)\neq \emptyset$, cannot be dispensed with. Let $X$ be any real vector space and $u\in X\setminus \{0\}.$ Define $K:=\{\lambda u\mid \lambda\geq 0\}$. Then $T\in \pi(K)\cap L_{inv}(X)$ if and only if $T(u)=\lambda u$ for some $\lambda>0.$ Then $T^{-1}(u)=\frac{1}{\lambda} u$, and so $T^{-1}$ is also positive. 
\end{enumerate}
\end{remark}

\begin{remark}
    The converse of Theorem \ref{thm::main} is: for every $T\in (\pi(K)\cap L_{inv}(X))\setminus Aut(K)$, there exist $S\in Aut(K), f\in K'$ and $u\in K$ such that $T$ is given by \eqref{genT}. This is not true in general, as we show next. 
    
    Let $X:=\R^2$ and $K:=\R^2_+$. It is well known that $S\in Aut(X)$ if and only if $S$ is a product of a permutation and a diagonal matrix with positive diagonals. Suppose that 
    \[
        \pi(K)\cap L_{inv}(X)\ni T:=\begin{bmatrix}
        1&3\\
        2&4
    \end{bmatrix} = P D + u v^\top
    \]
    for some permutation matrix $P$, a diagonal matrix $D$ with positive diagonal entries and $u,v\in K$. There are only two permutations to consider. Let $P=I,$ the identity matrix. Then, we have
    \[
        \begin{bmatrix}
        1&3\\
        2&4
    \end{bmatrix}=\begin{bmatrix}
        d_1&0\\
        0&d_2
    \end{bmatrix} + \begin{bmatrix}
        u_1v_1& u_1v_2\\
        u_2v_1& u_2v_2
    \end{bmatrix}.
    \]
    Since $d_i>0$ and $u,v\in K$, we get $u_1v_1< 1$ and $u_2v_2<4$. Hence $u_1v_1u_2v_2=(u_1v_2)(u_2v_1)<4,$ a contradiction to the computation done for the product of the off-diagonal entries. Now, let $P$ be the nonidentity permutation matrix and
     \[
        \begin{bmatrix}
        1&3\\
        2&4
    \end{bmatrix}=\begin{bmatrix}
        0&d_2\\
        d_1&0
    \end{bmatrix} + \begin{bmatrix}
        u_1v_1& u_1v_2\\
        u_2v_1& u_2v_2
    \end{bmatrix}.
    \]
    Solving the four equations, we get $u_1+(d_2-3) u_2=0$ and $(d_1-2)u_1+ u_2=0.$ This system has a nonzero solution $u_1,u_2$ if and only if $(d_2-3)(d_1-2)=1$. This is not possible, since $(d_2-3)(d_1-2)=u_1u_2v_1v_2=4$.
\end{remark}

\section{Cones with (possibly) empty interior}

In this section, we consider cones which may have no interior points. Recall that $\R_+$ denotes the set of all nonnegative real numbers.

\begin{theorem}\label{thm::ext_non_empty}
    Let $K\subseteq X$ be a cone in a real vector space $X$. Let $\dim (\text{span}(K))\geq 2$ and $\emph{ext}(K)\neq\{0\}$. Let $f\in K'$ be such that $f(v)\neq 0$ for some $v\in \emph{ext}(K).$ Then, for any $S\in Aut(K)$ and $u\in K\setminus \R_+ S(v)$, the map $T\colon X\to X$ given by \eqref{genT} is positive, invertible but $T^{-1}$ is not positive.
\end{theorem}
\begin{proof}
    Recall that $T\colon X\to X$ is defined as $T(x)=S(x)+f(x)u$, $x\in X.$  By Theorem \ref{thm::invertibility}, it suffices to show that $T^{-1}$ is not positive. Let $v\in \text{ext}(K)$ be such that $f(v)\neq 0$. Since $S\in Aut(K)$, by Lemma~\ref{lem::inv_image_of_extremal_is_extremal}, $S(v)\in \text{ext}(K)$. If $T(v)=S(v)+f(v)u\in \text{ext}(K)$, then we get $S(v)=\lambda T(v)$ for some $\lambda\in (0,1)$, and hence $u\in \R_+S(v)$, a contradiction. Therefore $T(v)\notin \text{ext}(K)$, thus, by Lemma~\ref{lem::inv_image_of_extremal_is_extremal}, we get $T\notin Aut(K)$.
\end{proof}

\begin{remark} The following observations are pertinent.
\begin{enumerate}
    \item
        The assumption on $f$ in Theorem~\ref{thm::ext_non_empty} is not really restrictive. In fact, if $X$ is a normed vector space with closed cone $K$, then for any $x\in K\setminus\{0\}$, there exists $f\in K'$ such that $f(x)\neq 0$. For, let $f(x)=0$ for every $f\in K'$. Then, by Proposition~\ref{prop::xinKifff(x)>0}, both $x$ and $-x$ are positive, so that $x=0$.
    \item
        Since $\dim (\text{span}(K))\geq 2$, the existence of $u\in K\setminus \R_+ S(v)$ is guaranteed.
    \item
        The assumption of $f$ being non-zero on some extremal is, in general, not equivalent to $f\in K'\setminus \{0\}$. For example, let $X:=\R^2$ and $K:=\{(x,y)\mid x>0~\text{or}~ x=0, y\geq 0\}$. $K$ is called the {\it lexicographic cone}. It is easy to see that $\text{ext}({K})=\{(0,y)\mid y\geq 0\}$. Consider $f\colon X\to \R$ defined as $f(x,y)=x$. Then, $f\in K'$, but $f[\text{ext}(K)]=\{0\}$. However, if $K=\text{hull}(\text{ext}(K))$\footnote{For any ordered vector space $X$ and $\emptyset\neq A\subseteq X$, $\text{hull}(A):=\{\sum_{i=1}^n\lambda_i a_i\mid n\in \N,~ \lambda_i\geq 0,~a_i\in A\}$.} and $\text{span}(K)=X$, then every nonzero positive functional is positive on some extremal. The condition $K=\text{hull}(\text{ext}(K))$ is automatically satisfied, if $X$ is a finite dimensional ordered vector space with a closed cone $K$ such that $\text{span}(K)=X$, due to the Krein-Milman theorem.
\end{enumerate}
\end{remark}

\begin{theorem}\footnote{This result and its proof is suggested by Prof. Dr. Jochen Glück during the discussion of topics from this paper.}\label{thm::Banach-space-case}
    Let $X$ be an ordered Banach space with a closed cone $K$ and $\dim(\text{span}(K)\geq 2)$. Let $S\in Aut(K)$. Then there exists $u\in K$ and $f\in K'$ such that the map $T\colon X\to X$ as in \eqref{genT} is positive, invertible but $T^{-1}$ is not positive.
\end{theorem}
\begin{proof}
    By \cite[Theorem 2.10]{alomost-interior-Jochen}, there exists $u\in K\setminus\{0\}$ and $f\in K'\setminus\{0\}$ such that $f(u)=0$. For $n\in \N$, define $T_n\colon X\to X$ as
    \[
        T_n(x)= S(x) + n (S'f)(x) u, \quad \forall x\in X,
    \]
    where $S'\colon X'\to X'$ is the algebraic dual operator, i.e. $S'\colon f\mapsto f\circ S$. Since, $S\in Aut(K)$, it follows that $S'f\in K'\setminus\{0\}$. Hence, by Theorem~\ref{thm::invertibility}, we get that $T_n$ is positive and invertible for each $n\in \N$. Further, it can be verified that
    \[
        T_n^{-1}(x)= S^{-1}(x) - nf(x) S^{-1}u, \quad \forall x\in X.
    \]
    Since $f\neq 0$, there exists $y\in K$ such that $f(y)>0$. If $T_n^{-1}(y)\geq 0$ for every $n\in \N$, then
    \[
        n (f(y)S^{-1}u)\leq S^{-1}y, \quad \forall n\in \N.
    \]
    Hence, by \eqref{eq::Archimedean}, we get $S^{-1}u\leq 0$. This is the contradiction to $u\in K\setminus \{0\}$, $S^{-1}\geq 0$ and $K\cap(-K)=\{0\}$. Hence, there exists $N\in \N$ such that $T_N^{-1}$ is not positive. This completes the proof.
\end{proof}

\section{Illustrative examples}

\begin{example}\label{hyp} 
Let $(\calH, \langle \cdot ,\cdot\rangle_{\calH})$ be a Hilbert space. Consider the space 
\[
    \calH\oplus \R:=\{(x,\alpha)\mid x\in \calH, ~\alpha\in \R\},
\]
called a {\it spin factor} of $\calH$. $\calH \oplus \R$ is again a Hilbert space with the inner product $\langle (\alpha, x), (\beta, y)\rangle:=\alpha \beta +\langle x,y\rangle_{\calH}$. Define $$K:=\{(x, \alpha)\in \calH\oplus \R \mid \alpha\geq \|x\|\}.$$ It is well known \cite[Page 43]{Anke_book} that $K$ is a closed cone with non-empty interior (e.g. $({\bf 0},1)\in \text{int}(K)$). $K$ is called the {\it hyperbolic cone or the Lorentz cone.}

Let ${\bf 0}\neq \hat{x}\in \calH$ be such that $\|\hat{x}\|=1$, and fix $u:=({\bf 0},1)$. It is clear that $(\hat{x}, 0)\ngeq {\bf 0}$ and $(\hat{x}, 0)\nleq {\bf 0}$. 
Consider the linear map $f\colon \calH\oplus \R\to \R$ defined as 
\[
    f(x,\alpha):=\alpha + \left\langle x, \hat{x}\right\rangle_{\calH}.
\]
Let $(x,\alpha)\in K$. By the Cauchy-Schwarz inequality, we get $|\langle x, \hat{x}\rangle_\calH|\leq \|x\|\leq \alpha$, so that $f(x,\alpha)\geq 0$. Thus, $f\in K'$. Note that, $f(\hat{x},0)=1$. Consider the map $T\colon \calH\oplus \R\to \calH\oplus \R$ defined  as 
\[
    T(x,\alpha)=(x,\alpha)+f(x,\alpha)({\bf 0},1)=(x, \alpha + f(x,\alpha)).
\]
By Theorem \ref{thm::invertibility}, $T$ is positive and invertible. The fact that $T(\hat{x}, 0)=(\hat{x}, 1) \in K$ with $(\hat{x},0)\notin K,$ shows that $T^{-1}$ is not positive.
\end{example}

\begin{example}
Let $X:=(C[0,1], \|\cdot \|_\infty)$ be the space of all continuous real valued functions on $[0,1]$, and $K:=C_+[0,1]\subseteq X$ be the cone of all point-wise nonnegative continuous functions. Then, $\text{int}(K)\neq \emptyset,$ while $\text{ext}(K)=\{0\}.$ Let $u:=\mathbbm{1}$ denote the constant $1$ function, so that $u\in \text{int}(K)$.  Let $f\in K'$ be defined by $f(x)=\int_{0}^1 x(t) dt$, $x\in X$. Let $T\colon C[0,1]\to C[0,1]$ be defined as 
\[
    T(x)(t)=x(t)+f(x)u(t), \quad t\in [0,1], ~x\in C[0,1].
\]
By Theorem \ref{thm::invertibility}, $T$ is positive and invertible. We claim that $T^{-1}$ is not positive. Indeed, consider $x\in X$ defined as follows:

\begin{minipage}{0.5\textwidth}
    \[
    x(t):=\begin{cases}
        20t, &0\leq t\leq \frac{1}{4}\\
        10-20t, &\frac{1}{4}\leq t\leq \frac{1}{2}\\
        2-4t, & \frac{1}{2}\leq t\leq \frac{3}{4}\\
        4t-4, &\frac{3}{4}\leq t\leq 1.
    \end{cases}
\]
\end{minipage}
\hspace{0.1cm}
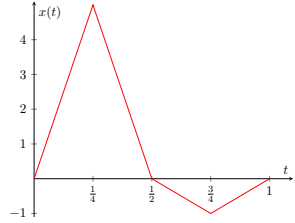
\begin{minipage}{0.5\textwidth}
\begin{figure}[H]
\centering
   \begin{tikzpicture}[scale=0.5]
        \begin{axis}[axis lines=middle, ymax=5.1, ymin=-1.1, xmax=1.1, xmin=0, xtick={0,1/4,1/2,3/4,1}, xticklabels={$0$,$\frac{1}{4}$,$\frac{1}{2}$,$\frac{3}{4}$,$1$}, ytick={-1,0,1,2,3,4}, yticklabels={$-1$, $0$, $1$, $2$, $3$, $4$}, clip=false, ylabel=$x(t)$, xlabel=$t$]
        \draw[color=red]{(0,0)--(1/4,5)--(1/2,0)--(3/4,-1)--(1,0)};
    \end{axis}
    \end{tikzpicture}
\end{figure}
\end{minipage}

\vspace{2pt}

Then $x\geq -\mathbbm{1}, ~f(x)=1, ~x\ngeq 0$ and $x\nleq 0$. 
Thus, $T(x)=x+u\in K,$ with $x\notin K$, showing that $T^{-1}$ is not positive.    
\end{example}

\begin{example}[The copositive and the positive semidefinite cones]

For $n\geq 2$, let $S^n$ be the space of all $n\times n$ real symmetric matrices. The cone $K_1:=\{A\in S^n\mid \langle Ax,x\rangle \geq 0,~\forall x\in \R^n\}$ is called the {\it Loewner cone} or the {\it positive semidefinite} cone, and $K_2:=\{A\in S^n\mid \langle Ax,x\rangle \geq 0,~\forall x\in \R^n_+\}$ is called the {\it copositive cone}. Let $u:=I$ denote the identity matrix. It is not difficult to see that $I\in \text{ int}(K_1)\cap \text{int}(K_2)$. 
Let $f\colon S^n\to S^n$ be the trace functional. It is clear that $f\in K_1'\cap K_2'$. Consider the map $T\colon S^n\to S^n$ defined as $T(A)=A+f(x)I$ for every $A\in S^n$. By Theorem~\ref{thm::invertibility}, for both the cones $K_1$ and $K_2$, $T$ is positive and invertible. If $D:=\text{diag}(-\frac{1}{2},1,\ldots, 1)\in S^n,$ then $D\notin K_1\cup K_2$ and $T(D)\in K_1\cap K_2$, proving that $T^{-1}$ is not positive. 

\end{example}


\begin{example}\label{lp}
Consider the space $(\ell^p, \|\cdot \|_p)$ of $p$ summable real sequences with the cone $\ell^p_+$ of all sequences in $\ell^p$ with nonnegative coordinates. Let $e_i$ denote the $i^{th}$ standard unit vector. It is well known that $\text{int}(\ell^p_+)=\emptyset$. However, we include a proof for ready reference. Let $x=(x_n)\in \ell^p_+$ and $\epsilon>0$. There exists $m\in \N$ such that $|x_m|<\frac{\epsilon}{2}$. Define $\hat{x}:= x- (2x_m) e_m \in \ell^p$. Then $\|x-\hat{x}\|=2 |x_m|~<\epsilon$, but $\hat{x}\ngeq 0$. Hence $x\notin \text{int}(\ell^p_+)$. Also, $\text{ext}(\ell^p_+)=\{\lambda e_i\mid \lambda\geq 0, ~i\in \N\}$. For, let $x=(x_n)\in \ell^p_+$ be such that $x_i\neq 0$ and $x_j\neq 0$ for $i\neq j$. Then $\frac{x_i}{2}e_i\leq x$ but $e_i$ and $x$ are linearly independent.  

Let $0\neq f\colon \ell^p\to \R$ be a positive functional. Let $f(e_i)\neq 0$ for some $i\in \N$. For $u:=e_j\in \ell^p_+$ with $j\neq i$, define $T\colon \ell^p\to \ell^p$, by $T(x)=x+f(x)u$, $x\in \ell^p$. By Theorem~\ref{thm::invertibility}, $T$ is positive and invertible. By applying Remark \ref{spform}, $T^{-1}e_i= e_i - \lambda e_j$, where $\lambda=\frac{f(e_i)}{1+f(e_j)}$, proving that $T^{-1}$ is not positive.
\end{example}

Our concluding example considers a  ordered vector space where both $\text{int}(K)$ and $\text{ext}(K)\setminus \{0\}$ are empty.

\begin{example}
Let $X:=(C_0(\R), \|\cdot \|_\infty)$ be the space of all continuous functions vanishing at $\pm \infty$, and $K\subseteq X$ be the cone of all point-wise nonnegative functions. It is not difficult to see that $\text{int}(K)=\emptyset$ and $\text{ext}(K)=\emptyset$. Hence, Theorem \ref{thm::main} and Theorem \ref{thm::ext_non_empty} are not applicable. However, we can still construct $T\in \pi(K)\cap L_{inv}(X)\setminus Aut(K)$ using our method.

Let $f\colon X\to \R$ be the point evaluation functional at $0$, and $u\in X$ be defined as $u(t)=e^{-|t|}, ~t\in \R$. Define the linear map $T\colon X\to X$ as 
\[
    Tx(t)=(x+f(x)u)(t)=x(t)+x(0)u(t), \quad x\in X, ~t\in \R.
\]
By Theorem~\ref{thm::invertibility}, $T$ is positive and invertible. To show that $T^{-1}$ is not positive, it suffices to show that there exists $x\in X$ such that $x(0)=1$, $x\ngeq 0$, $x+u\geq 0$. Given below is one such $x$.
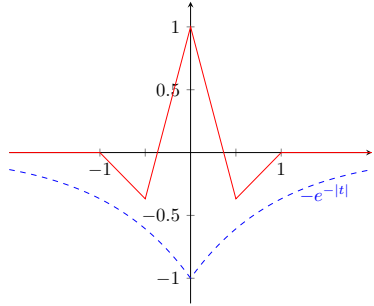
\begin{figure}[H]
    \centering
    \begin{tikzpicture}[scale=0.7] 
\begin{axis}[axis lines=middle, ymax=1.2, ymin=-1.2, xmax=2, xmin=-2, xtick={-1,-1/2,0,1/2,1}, xticklabels={$-1$, , , , $1$}, clip=false, domain=-2:2]
    
        \addplot[color=blue, samples=1000, dashed]{-exp(-abs(x)} node[right,pos=0.8]{$-e^{-|t|}$};
        \draw[color=red]{(-2,0)--(-1,0)--(-1/2, -1/e)--(0,1)--(1/2,-1/e)--(1,0)--(2,0)};
    \end{axis}
\end{tikzpicture}
\label{fig:placeholder}
\end{figure}

\end{example}

\textbf{Acknowledgment:} 
The authors thank Prof. Apoorva Khare for carefully reading the manuscript and providing useful suggestions. Pavankumar thanks Prof. Dr. Jochen Glück for discussing the topics of this article. The first author acknowledges funding received from the Prime Minister’s Research Fellowship (PMRF), Ministry of Education, Government of India, for his entire PhD duration.

\end{document}